\documentclass[12pt,letterpaper]{article}

\usepackage{amssymb,amsfonts,amscd,amsthm}
\usepackage[all,arc]{xy}
\usepackage{enumerate, bbm}
\usepackage{mathrsfs}
\usepackage{tikz}
\usepackage{mathtools}
\usepackage{hyperref}
\usepackage{color}
\usepackage{graphicx}
\usepackage{enumitem} 
\graphicspath{ {TexImages/} }

\newtheorem{thm}{Theorem}[section]
\newtheorem{cor}[thm]{Corollary}
\newtheorem{prop}[thm]{Proposition}
\newtheorem{lem}[thm]{Lemma}

\theoremstyle{definition}

\theoremstyle{remark}

\hypersetup{
	colorlinks,
	citecolor=blue,
	filecolor=blue,
	linkcolor=blue,
	urlcolor=blue,
	linktocpage
}

\setenumerate[1]{label=\thesection.\arabic*.} 
\setenumerate[2]{label*=\arabic*.} 
\setlist[enumerate]{itemsep=2ex, topsep=2ex} 
\setlist[itemize]{itemsep=2ex, topsep=2ex}

\renewcommand{\l}{\left}
\renewcommand{\r}{\right}

\newcommand{\sm}{\setminus}
\newcommand{\sub}{\subseteq}

\newcommand{\asc}{\mr{asc}}

\def\inv{\mr{inv}}

\newcommand{\Eul}{\genfrac{\langle}{\rangle}{0pt}{}}
\renewcommand{\c}[1]{\mathcal{#1}}
\renewcommand{\b}[1]{\mathbf{#1}}

\newcommand{\mr}[1]{\mathrm{#1}}

\newcommand{\SP}{\textbf{\textcolor{orange}{SKIP}}}

\renewcommand{\SP}{\mr{SPF}}
\newcommand{\kSP}{\SP_k}
\newcommand{\OPF}{\mr{OPF}}
\newcommand{\IPF}{\mr{IPF}}
\newcommand{\PF}{\mr{PF}}
\title{Subset Parking Functions}

\date{\today}
\author{Sam Spiro\footnote{Dept.\ of Mathematics, UCSD {\tt sspiro@ucsd.edu}.}}
\begin{document}
	\maketitle
\begin{abstract}
	A parking function $(c_1,\ldots,c_n)$ can be viewed as having $n$ cars trying to park on a one-way street with $n$ parking spots, where car $i$ tries to park in spot $c_i$, and otherwise he parks in the leftmost available spot after $c_i$.  Another way to view this is that each car has a set $C_i$ of ``acceptable'' parking spots, namely $C_i=[c_i,n]$, and that each car tries to park in the leftmost available spot that they find acceptable.
	
	Motivated by this, we define a subset parking function $(C_1,\ldots,C_n)$, with each $C_i$ a subset of $\{1,\ldots,n\}$, by having the $i$th car try to park in the leftmost available element of $C_i$.  We further generalize this idea by restricting our sets to be of size $k$, intervals, and intervals of length $k$.  In each of these cases we provide formulas for the number of such parking functions.
\end{abstract}
\section{Introduction}

Parking functions are a well studied object in combinatorics, and are often defined in the following way.  Imagine that there are $n$ parking spots labeled 1 though $n$ on a one way street.  There are $n$ cars, also labeled 1 through $n$, that wish to park in these spots, and each has a preferred parking spot $c_i$.  When it is car $i$'s turn to park, he goes to his preferred spot $c_i$ and parks there if it is empty.  Otherwise, he tries to park in the next available spot that is after $c_i$.  The tuple $(c_1,\ldots,c_n)$ is said to be a parking function if every car succeeds in parking.

For example, $(2,2,2)$ is not a parking function, as car 1 parks in slot 2; car 2 tries to park in slot 2 but can not and goes to the next available slot 3; and car 3 tries to park in slot 2 but can not, and there are no slots after this that are available.  On the other hand, $(2,1,1)$ is a parking function as car 1 parks in slot 2; car 2 parks in slot 1; and car 3 tries to park in slot 1 but can not, so he goes to the next available slot 3.  We will say that the parking function $(2,1,1)$ has outcome $213$, which describes how one would see the cars parked if one were to walk from slot 1 to slot $3$.

Parking functions are well studied and have many interesting combinatorial properties.  For example, a necessary and sufficient condition for $(c_1,\ldots,c_n)$ to be a parking function is, after rearranging the $c_i$ in increasing order as $b_1\le b_2\le \cdots\le b_n$, we have $b_i\le i$ for all $i$.  In particular, this shows that any permutations of the entries of a parking function is also a parking function.  One can also prove that the number of parking functions $\PF(n)$ satisfies \begin{equation}\label{E-Park}
\PF(n)=(n+1)^{n-1}.
\end{equation}
Parking functions have connections to many other areas of combinatorics, such as hyperplane arrangements \cite{Stan79} and the lattice of non-crossing partitions \cite{Stan80}.  We refer the reader to the survey of Yan \cite{CY} for an elegant proof of \eqref{E-Park} and a more in depth study of parking functions.

Many generalizations and variants of parking functions have been studied, such as $\b{x}$-parking functions \cite{Yx} and $G$-parking functions \cite{PS}. These examples generalize the $b_i\le i$ characterization of parking functions.  One can also generalize the parking analogy.  An example of this is to allow cars to park a few spaces before their preferred spot if this is already taken, which has been studied recently \cite{Naples}.

In this paper we also consider a variant of parking function that is obtained by modifying the parking rule.  To motivate the idea, we observe that a parking function $(c_1,\ldots,c_n)$ can be viewed as each car choosing a set $C_i=[c_i,n]$ of ``acceptable'' parking spaces, with each car parking in the leftmost available spot which is acceptable to them.  One can generalize this idea by allowing each $C_i$ to be an arbitrary set.  

To this end, let $[n]:=\{1,2,\ldots,n\}$ and let $\c{S}_n$ denote the set of permutations of size $n$ written in one line notation.  Given $n$ non-empty subsets $C_i\sub [n]$ and a permutation $\pi\in \c{S}_n$, we will say that $\c{C}=(C_1,\ldots,C_n)$ is a subset parking function with outcome $\pi$ if for all $1\le i\le n$, having $\pi_j=i$ implies $j$ is the smallest element of $C_i\sm\{\pi_{i'}^{-1}:i'<i\}$. That is, if car $i$ ends up in spot $j$, it must find spot $j$ to be acceptable, all the earlier spots which are acceptable are already taken, and no one has taken spot $j$ yet.  We let $\SP(n,\pi)$ denote the number of subset parking functions with outcome $\pi$, and we denote the total number of subset parking functions by $\SP(n):=\sum_{\pi\in \c{S}_n} \SP(n,\pi)$.  Technically we should say that $\SP(n)$ counts the number of parking functions of size $n$, but here and throughout we omit explicitly noting this dependency on $n$ whenever it is clear from context.

For example, $(\{2\},\{2,3\},\{1,2,3\})$ is a subset parking function with outcome $312$.  However, $(\{2,3\},\{1,2,3\},\{2\})$ is not a subset parking function since we require $\pi_2=1,\ \pi_1=2$, and then no choice from $C_3$ will work.  In particular this shows that for subset parking functions the order of the $C_i$ sets are important, which is not the case in the classical study of parking functions.

Our first goal is to enumerate subset parking functions.  To this end, we recall that the inversion number $\inv(\pi)$ of a permutation $\pi$ is equal to the number of pairs $(i,j)$ such that $i<j$ and $\pi_j<\pi_i$.

\begin{thm}\label{T-SP}
	For any integer $n\ge 1$ and $\pi\in \c{S}_n$,
	\[\SP(n,\pi)=2^{n(n-1)-\inv(\pi)},\]
	\[\SP(n)=\prod_{i=0}^{n-1}(2^n-2^i).\]

\end{thm}

(Classical) parking functions are subset parking functions where each $C_i$ is required to be an interval of the form $[c_i,n]$.  We can get other interesting variants by restricting the $C_i$ sets in other ways. For example, we say that $\c{C}=(C_1,\ldots,C_n)$ is a $k$-subset parking function if $\c{C}$ is a subset parking function and $|C_i|=k$ for all $i$.  We let $\kSP(n,\pi)$ denote the number of $k$-subset parking functions with outcome $\pi$ and $\kSP(n)$ the number of $k$-subset parking functions.

To state our next result, we define the local inversion number $\inv_i(\pi)$ to be the number of pairs $(i,j)$ with $i<j$ and $\pi_j<\pi_i$.  Observe that $\inv(\pi)=\sum \inv_i(\pi)$.  We adopt the convention that ${0\choose 0}=1$ and ${0\choose x}=0$ for $x>0$.

\begin{thm}\label{T-k}
	For any integer $n\ge 1$, $\pi\in \c{S}_n$, and $1\le k\le n$,
	\[\kSP(n,\pi)=\prod_{i=1}^{n} {n-\inv_{i}(\pi)-1\choose k-1},\]
	\[\kSP(n)=\prod_{i=0}^{n-1}\l({n\choose k}-{i\choose k}\r).\]

\end{thm}

We next consider the case that each $C_i$ is an interval, and we call such parking functions interval parking functions. Let $\IPF(n)$ and $\IPF(n,\pi)$ be the total number of interval parking functions and the number of interval parking functions with outcome $\pi$, respectively.

To state our full result, given a permutation $\pi$ we define $a_i(\pi)$ to be the largest $j$ with $1\le j\le i$ such that $\pi_i\ge  \{\pi_{i},\pi_{i-1},\ldots,\pi_{i-j+1}\}$.  For example, $a_i(\pi)\ge 2$ if and only if $\pi_i>\pi_{i-1}$.  As another example, for $\pi=31524$, we have $a_i(\pi)$ equal to 1, 1, 3, 1, 2 as $i$ ranges from 1 to 5.  Finally, define $\PF(n,\pi)$ to be the number of (classical) parking functions with outcome $\pi$.

\begin{thm}\label{T-I}
	For any $n\ge 1$ and $\pi\in \c{S}_n$, 
	\[\IPF(n,\pi)=n!\cdot \PF(n,\pi)=n! \prod_{i=1}^n a_i(\pi),\]
	\[\IPF(n)=n!\cdot \PF(n)=n!\cdot (n+1)^{n-1}.\]
\end{thm}
Other properties of interval parking functions are currently being investigated by Christensen, DeMuse, Martin, and Yin \cite{CDM}.

The last variant we consider are $k$-interval parking functions, which are interval parking functions where each $C_i$ is an interval containing $k$ elements.  We let $\IPF_k(n)$ denote the number of $k$-interval parking functions and $\IPF_k(n,\pi)$ the number of those with outcome $\pi$.  Define $\c{S}_n^k$ to be the set of permutations $\pi$ of order $n$ with $\pi_n>\pi_{n-1}>\cdots >\pi_{n-k+1}$.

\begin{thm}\label{T-kI}
	Let $k$ and $n$ be integers with $1\le k\le n$ and let $\pi\in \c{S}_n$. If $\pi\notin \c{S}_n^{k}$, then $\IPF_k(n,\pi)=0$.  Otherwise,\[\IPF_k(n,\pi)=\prod_{i=1}^{n-k} \min\{a_i(\pi),k\}\cdot \prod_{i=n-k+1}^{n}\min\{n-i-k+a_{i}(\pi)+1,n-i+1\}.\]
\end{thm}

This formula is rather complicated, but for certain $k$ it is manageable.  For example, when $k=1$ each term in the product is 1.  We conclude that $\IPF_1(n,\pi)=1$ for all $\pi$, and hence $\IPF_1(n)=n!$.  When $k=n$, we have $\IPF_n(n,\pi)=1$ when $\pi=12\cdots n$ (since in general $a_1(\pi)=1$ and $a_{n-i}(\pi)\le n-i$), and otherwise $\IPF_n(n,\pi)=0$, so $\IPF_n(n)=1$.  Both of these results can also be verified directly.  The formulas for $k=n-1$ and $k=2$ are also quite nice.  

\begin{cor}\label{C-n-1}
	If $n\ge 2$ and $\pi\in S^{n-1}_n$ with $\pi_1=j$, then \[\IPF_{n-1}(n,\pi)=\begin{cases*}
	2^{n-j-1}, & $j\ne n$,\\ 
	1 & $j=n$.
	\end{cases*}\]
	Moreover, \[\IPF_{n-1}(n)=2^{n-1}.\]
\end{cor}

To state the formula for $k=2$, we define the ascent number $\asc(\pi)$ of a permutation $\pi$ to be the number of $i$ with $2\le i\le n$ and $\pi_{i-1}<\pi_i$.  Define the Eulerian number $\Eul{n}{k}$ to be the number of permutations $\pi\in \c{S}_n$ with $\asc(\pi)=k$.

\begin{cor}\label{C-2}
	If $n\ge 2$ and $\pi\in S^2_n$, then \[\IPF_2(n,\pi)=2^{\asc(\pi)-1}.\]
	Moreover,
	\[\IPF_2(n)=\sum_{k=1}^{n-1} (n-k)\Eul{n-1}{k-1}2^{k-1}.\]
\end{cor}

\section{Subset Results}\label{S-Sub}
We first prove enumeration results for a generalization of subset parking functions where each car is given a list of allowed subset sizes.  To this end, given $\c{L}=(L_1,\ldots,L_n)$ with $L_i\sub [n]$, we define $\SP(n,\c{L})$ to be the number of subset parking functions where $|C_i|\in L_i$, and we will call this an $\c{L}$-parking function.  Our first goal will be to enumerate $\SP(n,\c{L})$.

To do this, we define the notion of a partial parking function, which intuitively describes where the first $m$ cars have parked.  Let $S_{m,n}$ denote the set of strings $\pi=\pi_1\cdots \pi_n$ where for all $1\le i\le m$ there exists a unique index $j$ with $\pi_j=i$ and such that every other letter is an auxillary letter $*$.  Note that $S_{n,n}$ is simply the set of permutations.  For $i\le m$ we let $\pi_i^{-1}$ denote the unique index $j$ with $\pi_j=i$.  

We say that $(C_1,\ldots,C_m)$ with each $C_i$ a non-empty subset of $[n]$ is a partial $\c{L}$-parking function with outcome $\pi\in S_{m,n}$ if for all $1\le i\le m$, $|C_i|\in L_i$ and $\pi_i^{-1}=j$ implies $j$ is the smallest element of $C_i\sm\{\pi_{i'}^{-1}:i'<i\}$.  Finally, given a permutation $\pi$, we write $\pi^{(m)}$ to denote the string where $\pi_i^{(m)}=\pi_i$ if $i\le m$ and $\pi_i^{(m)}=*$ otherwise.  Once one unpacks these definitions, the following is immediate.

\begin{lem}\label{L-PP1}
	Let $1\le m\le n$ and $\pi\in \c{S}_n$.  $\c{C}$ is an $\c{L}$-parking function with outcome $\pi$ if and only if $(C_1,\ldots,C_m)$ is a partial $\c{L}$-parking function with outcome $\pi^{(m)}$ for all $m$.
\end{lem}
The following lemma shows how to extend partial parking functions.

\begin{lem}\label{L-PP2}
	Let $1\le m\le n$ and $\pi\in \c{S}_n$. If $(C_1,\ldots,C_{m-1})$ is a partial $\c{L}$-parking function with outcome $\pi$, then $(C_1,\ldots,C_{m})$ is a partial $\c{L}$-parking function if and only if $|C_{m}|\in L_{m}$ and $C_{m}\not \subset \{\pi_{i}^{-1}:i<m\}$.
\end{lem}
\begin{proof}
	If $C_{m}$ is such a set, then by assumption $|C_{m}|\in L_{m}$ and there exists some minimal $j$ in $C_{m}\sm \{\pi_{k}^{-1}:k<m\}$.  Thus by defining $\pi'$ by $\pi'_i=\pi_i$ for $i\ne j$ and $\pi'_j=m$, we see that $(C_1,\ldots,C_{m})$ is a partial $\c{L}$-parking function with outcome $\pi'$.  Conversely, if $(C_1,\ldots,C_{m})$ is a partial $\c{L}$-parking function, then $C_{m}\sm \{\pi_{k}^{-1}:k<m\}$ must be non-empty, so $C_{m}\not \subset \{\pi_{k}^{-1}:k<m\}$.  We also must have $|C_{m}|\in L_{m}$ by definition, proving the result.
\end{proof}

\begin{thm}\label{T-Tech}
	For any $n\ge 1$ and $\c{L}=(L_1,\ldots,L_n)$,
	\[\SP(n,\c{L})=\prod_{i=1}^n\l(\sum_{\ell \in L_i} {n\choose \ell}-{i-1\choose \ell}\r).\]
\end{thm}
\begin{proof}
	Consider the following procedure.  We start with an empty list $()$.  Recursively, given a partial $\c{L}$-parking function $(C_1,\ldots,C_{i-1})$, we choose a set $C_i$ such that $(C_1,\ldots,C_{i})$ is a partial $\c{L}$-parking function.  By Lemma~\ref{L-PP1}, every $\c{L}$-parking function is obtained (uniquely) by this procedure.  Thus to obtain our result we need only enumerate how many choices we can make at each stage of the procedure.
	
	Assume one has already chosen $(C_1,\ldots,C_{i-1})$ so now we need to choose $C_i$.  By Lemma~\ref{L-PP2}, for any $\ell\in L_i$, the number of ways to choose an appropriate $C_i$ with $|C_i|=\ell$ is ${n\choose \ell}-{i-1\choose \ell}$.  Namely, one can choose any $\ell$-element subset that is not contained in $\{\pi_j^{-1}:j<i\}$.  As we allow $|C_i|$ to be any element of $L_i$, we conclude that the number of choices for $C_i$ is exactly $\sum_{\ell \in L_i} {n\choose \ell}-{i-1\choose \ell}$.   As the number of choices for $C_i$ is independent of all of the other $C_j$ sets, we conclude that the total number of ways to complete this procedure is the product of all of these sums.  This gives the desired result.
\end{proof}

We can prove a similar general theorem when the outcome is specified.  To this end, define $\SP(n,\c{L},\pi)$ to be the number of $\c{L}$-parking functions with outcome $\pi$.  Recall that $\inv_i(\pi)$ is defined to be the number of $(i,j)$ with $i<j$ and $\pi_j<\pi_i$.

\begin{thm}\label{T-TechPi}
	For any $n\ge 1$, $\pi\in \c{S}_n$, and $\c{L}=(L_1,\ldots,L_n)$,
	\[\SP(n,\c{L},\pi)=\prod_{i=1}^n\l(\sum_{\ell \in L_i} {n-\inv_i(\pi)-1\choose \ell-1}\r).\]
\end{thm}
\begin{proof}
	Consider the following procedure.  We start with an empty list $()$.  Recursively, given a partial $\c{L}$-parking function $(C_1,\ldots,C_{i-1})$ with outcome $\pi^{(i-1)}$, we choose a set $C_i$ such that $(C_1,\ldots,C_{i})$ is a partial $\c{L}$-parking function with outcome $\pi^{(i)}$.  By Lemma~\ref{L-PP1}, every $\c{L}$-parking function with outcome $\pi$ is obtained (uniquely) by this procedure.  Thus to obtain our result we need only enumerate how many choices we can make at each stage of the procedure.
	
	Assume $(C_1,\ldots,C_{i-1})$ is a partial $\c{L}$-parking function with outcome $\pi^{(i-1)}$.  By Lemma~\ref{L-PP2}, if we wish to have $|C_i|=\ell\in L_i$, then we must have $C_{i}\not \subset \{\pi_{i'}^{-1}:i'<i\}$.  Moreover, we also must choose this set so that it has outcome $\pi^{(i)}$.  If $j=\pi_i^{-1}$, this is equivalent to having $C_{i}$ be any subset with $j$ the minimal element of $C_{i}\sm \{\pi_{i'}^{-1}:i'<i\}$.  To summarize, necessary and sufficient conditions for $C_i$ to have $|C_i|=\ell$ are
	\begin{itemize}
		\item[(a)] $|C_i|=\ell$, 
		\item[(b)] $C_i\not\subset\{\pi_{i'}^{-1}:i'<i\}$,
		\item[(c)] $j\in C_i$, and
		\item[(d)] $k\notin C_i$ if $k<j$ and $k\notin \{\pi_{i'}^{-1}:i'<i\}$.

	\end{itemize}  Note that (b) is implied by (c), so this is irrelevant.  Condition (d) is equivalent to avoiding $k$ with $k<j$ and $i<\pi_k$ (that is, the car that appears in the earlier spot $k$ parks after $i$).  The number of such $k$ is exactly $\inv_i(\pi)$, so we conclude that the number of $C_i$ satisfying these conditions is exactly ${n-1-\inv_i(\pi)\choose \ell-1}$.   Summing this value over all $\ell\in L_i$ gives the total number of choices for $C_i$.  As this quantity is independent of all the other choices of $C_j$, we can take their product to arrive at the desired count for $\SP(n,\c{L},\pi)$.
\end{proof}

With this we can now prove our results.  We start with Theorem~\ref{T-k}.

\begin{proof}[Proof of Theorem~\ref{T-k}]
	Note that $k$-subset parking functions are precisely $\c{L}$-parking functions where $L_i=\{k\}$ for all $i$.  The result follows from Theorems~\ref{T-Tech} and \ref{T-TechPi}.
\end{proof}

\begin{proof}[Proof of Theorem~\ref{T-SP}]
	Subset parking functions are precisely $\c{L}$-parking functions where $L_i=[n]\sm \{0\}$ for all $i$.  By Theorem~\ref{T-Tech} we have \[\SP(n)=\prod_{i=1}^n \l((2^n-1)-(2^{i-1}-1)\r).\]  Canceling the 1's and reindexing the product gives the first result.  For the second result, Theorem~\ref{T-TechPi} implies \[\SP(n,\pi)=\prod_{i=1}^n 2^{n-\inv_i(\pi)-1}=2^{n(n-1)-\inv(\pi)},\]
	where we used that $\inv(\pi)=\sum \inv_i(\pi)$.
\end{proof}

We note that Theorem~\ref{T-SP} implies \[\sum_{\pi \in \c{S}_n}2^{n(n-1)-\inv(\pi)}=\prod_{i=0}^{n-1} (2^n-2^i),\]
which one can verify using the generating function for the inversion statistic.  This also provides an alternative way to prove the formula for $\SP(n)$ given the formulas for each $\SP(n,\pi)$.  Similarly Theorem~\ref{T-k} implies 
\[
	\sum_{\pi \in \c{S}_n} \prod_{i=1}^n {n-\inv_i(\pi)-1\choose k-1}=\prod_{i=1}\l({n\choose k}-{i-1\choose k}\r).
\]
We are not aware of a more direct method to prove this.

Before closing this section, we briefly discuss a variant of subset parking functions.   Since subset parking functions allow each car to have any set of positions be acceptable, it also makes sense to allow each car to have their own preference order for these spots insteadof always requiring them to park in the left-most available spot.  

To formalize this, we say that a list of subsets $\c{C}=(C_1,\ldots,C_n)$, together with a list of bijections $f_i:C_i\to [|C_i|]$, is an ordered parking function with outcome $\pi=\pi_1\cdots \pi_n$ if for all $1\le i\le n$, $\pi_j=i$ implies $f_i(j)=\min_{j'\in D_i} f_i(j')$, where $D_i:=C_i\sm\{\pi_{i'}^{-1}:i'<i\}$.  We let $\OPF(n,\pi)$ denote the number of ordered parking functions with outcome $\pi$ and $\OPF(n):=\sum \OPF(n,\pi)$ the number of ordered parking functions.

If we define $\c{L}$-ordered parking functions analogous to how we defined $\c{L}$-parking functions, then essentially the same proof used to prove Theorem~\ref{T-Tech} shows that
\[
\OPF(n,\c{L})=\prod_{i=1}^n\l(\sum_{\ell \in L_i} {n\choose \ell}\ell!-{i-1\choose \ell}\ell!\r).
\] 
With this established, one can prove a nice analog of Theorem~\ref{T-SP}.  Namely, define $\c{O}(n)=\sum_{\ell=0}^n {n\choose \ell}\ell!$ to be the number of ordered subset of $[n]$.  Then
\[
\OPF(n)=\prod_{i=0}^{n-1}(\c{O}(n)-\c{O}(i)).
\]
In the ordered setting, every $\pi$ is equally likely to be the outcome of an ordered parking function, so $\OPF(n,\pi)=\OPF(n)/n!$ for all $\pi$.
\section{Interval Results}\label{S-Int}
As before we first prove a more general theorem.  Let $\c{K}=(K_1,\ldots,K_n)$ be such that $K_i\sub [n]\sm \{0\}$ for all $i$. We say that $(C_1,\ldots,C_n)$ is a $\c{K}$-interval parking function if each $C_i$ is an interval with $|C_i|\in K_i$.  We define partial $\c{K}$-interval parking functions analogous to how we defined partial $\c{L}$-parking functions in the previous section, and as before we immediately have the following.
\begin{lem}\label{L-k1}
	Let $1\le m\le n$ and $\pi \in \c{S}_n$.  $\c{C}$ is a $\c{K}$-interval parking function with outcome $\pi$ if and only if $(C_1,\ldots,C_m)$ is a partial $\c{K}$-interval parking function with outcome $\pi^{(m)}$ for all $m$.
\end{lem}

We also have an analog of Lemma~\ref{L-PP2}.  Recall that we define $a_i(\pi)$ to be the largest $j\le i$ such that $\pi_i\ge  \{\pi_{i},\pi_{i-1},\ldots,\pi_{i-j+1}\}$.
\begin{lem}\label{L-k2}
	Let $1\le m\le n$ and $\pi \in \c{S}_n$. Let $(C_1,\ldots,C_{m-1})$ be a partial $\c{K}$-interval parking function with outcome $\pi^{(m-1)}$ and let $p=\pi_m^{-1}$.  Then $(C_1,\ldots,C_{m})$ is a partial $\c{K}$-interval parking function with outcome $\pi^{(m)}$ if and only if $C_m=[r,r+k-1]$ with $k\in K_i$ and \[\max\{p-a_p(\pi)+1,\ p-k+1\}\le r\le \min\{p,\ n-k+1\}.\]
\end{lem}
\begin{proof}
	Assume $(C_1,\ldots,C_{m})$ is such a partial $\c{K}$-interval parking function with $C_m=[r,r+k-1]$ for some $r$ and $k$.  Because $|C_m|=k$ we require $k\in K_m$, and because $C_m\sub [n]$ we must have $r+k-1\le n$. We also need $p-k+1\le r\le p$ so that this set contains $p$.  Further, we require every $x\in [r,p]$ to satisfy $\pi_x<m$, otherwise $p$ will not be the smallest element of $C_m\sm \{\pi_i^{-1}:i<m\}$, which would contradict $(C_1,\ldots,C_m)$ having outcome $\pi^{(m)}$.  By definition this will not be the case if $r<p-a_p(\pi)+1$, so $r\ge p-a_p(\pi)+1$.  We conclude that $r$ satisfies the desired inequalities.
	
	Conversely, assume $C_m=[r,r+k]$ has $r$ and $k$ satisfying these conditions.  Because $p-a_p(\pi)+1\ge 1$ we have $C_m\sub [n]$, and we also have $|C_m|=k\in K_i$.  Again by definition of $a_p(\pi)$ these inequalities imply that $p$ is the smallest element of $C_m\sm \{\pi_i^{-1}:i<m\}$, so this gives the desired partial $\c{K}$-interval parking function.
\end{proof}

Let $\IPF(n,\c{K},\pi)$ denote the number of $\c{K}$-interval parking functions with outcome $\pi$ and define \[b_i(\pi,k):=\begin{cases*}
\min\{a_i(\pi),\ k\} & $i\le n-k$, \\ 
0 & $a_i(\pi)< k+i-n$,\\ 
\min\{n-i-k+a_i(\pi)+1,n-i+1\} & otherwise.
\end{cases*}\]
\begin{thm}\label{T-TechI}
	For any $n\ge 1,\ \pi\in \c{S}_n$, and $\c{K}=(K_1,\ldots,K_n)$, 
	\[\IPF(n,\c{K},\pi)=\prod_{i=1}^n \l(\sum_{k\in K_{\pi_i}} b_{i}(\pi,k)\r).\]
\end{thm}
\begin{proof}
	We consider the number of ways to iteratively build partial $\c{K}$-interval parking functions with outcomes $\pi^{(m)}$.  If one has already chosen $(C_1,\ldots,C_{i-1})$ and $p=\pi_i^{-1}$, then by Lemma~\ref{L-k2} the number of ways to choose an appropriate $C_i$ with $|C_i|=k\in K_i$ is the number of $r$ in the range \[\max\{p-a_p(\pi)+1,p-k+1\}\le r\le \min\{p,n-k+1\}.\]  If $p\le n-k$ this number is exactly $\min\{a_p(\pi),k\}$. Otherwise it is \[\max\{0,\min\{n-k-p+a_p(\pi)+1,n-p+1\}\}.\]  Because $n-p+1\ge 1$, this quantity is 0 if and only if $a_p(\pi)+1\le k+p-n$.  Thus the number of choices for $C_i$ with $|C_i|=k$ is exactly $b_p(\pi,k)$.  Summing this over all $k\in K_i$ gives a quantity independent of all the other $C_j$, so we can take the product of these values and conclude 
	\[\IPF(n,\c{K},\pi)=\prod_{i=1}^n \l(\sum_{k\in K_{i}} b_{\pi_i^{-1}}(\pi,k)\r).\]
	By reindexing this product, we get the stated result.
\end{proof}

\begin{proof}[Proof of Theorem~\ref{T-kI}]
	Recall that we wish to prove \[\IPF_k(n,\pi)=\prod_{i=1}^{n-k} \min\{a_i(\pi),k\}\cdot \prod_{i=n-k+1}^{n}\min\{n-i-k+a_{i}(\pi)+1,n-i+1\}\] whenever $\pi\in \c{S}_n^k$.  That is, whenever $\pi_n>\cdots>\pi_{n-k+1}$.  Observe that $k$-interval parking functions are exactly $\c{K}$-interval parking functions with $K_i=\{k\}$ for all $i$, so a formula for $\IPF_k(n,\pi)$ is given by Theorem~\ref{T-TechI}.  It remains to rewrite this formula into the desired form.
	
	If $\pi\notin \c{S}_n^k$, then there exists some $i$ with $0\le i\le k-2$ and $\pi_{n-i}<\pi_{n-i-1}$. This implies $a_{n-i}(\pi)=1<k-i$, and hence $b_{n-i}(\pi,k)=0$.  Thus $\IPF_k(n,\pi)=0$.
	
	From now on we assume $\pi \in \c{S}_n^k$.  This implies $a_{n-i}(\pi)\ge k-i$ for all $0\le i\le k-1$, and hence for these $i$ we have $b_{n-i}(\pi,k)=\min\{i-k+a_{n-i}(\pi)+1,i+1\}$.  This gives $b_j(\pi,k)$ for all $j\ge n-k+1$, and otherwise we have $b_j(\pi,k)=\min\{a_j(\pi),k\}$.  Taking the products of these terms gives the desired result. 
\end{proof}

\begin{proof}[Proof of Corollary~\ref{C-n-1}]
	The statement can be verified for $n=2$, so assume $n\ge 3$.  By Theorem~\ref{T-k} we have for $\pi \in S^{n-1}_n$ that \begin{align*}\IPF_{n-1}(n,\pi)&=\min\{a_1(\pi),n-1\}\cdot \min\{a_2(\pi),n-1\}\cdot \prod_{i=3}^{n} \min\{a_{i}(\pi)-i+2,n-i+1\}\\&=a_2(\pi)\cdot \prod_{i=3}^{n-1} \min\{a_{i}(\pi)-i+2,n-i+1\},\end{align*}
	where we used $a_1(\pi)\le 1$ and $a_2(\pi)\le 2\le n-1$.  We claim that this is equal to 
	\[\prod_{i=2}^{n-1} (a_i(\pi)-i+2).\]
	Indeed this follows from the fact that $a_i(\pi)-i+2\le 2\le n-i+1$ for all $i\le n-1$.  
	
	Assume $\pi_1=j$, and recall that $\pi\in \c{S}_n^{n-1}$ implies that $\pi_2<\cdots<\pi_n$. Thus for all $i>1$ we have $a_i(\pi)=i$ if $\pi_i>j$ and $a_i(\pi)=i-1$ otherwise.  Thus  $j=n$ implies that $a_i(\pi)=i-1$ for all $i\ge 2$, and otherwise there are exactly $n-1-j$ different $i$ with $2\le i\le n-1$ and $a_i(\pi)=i$.  We conclude the first result. For the second result,
	\[
	\IPF_{n-1}(n)=\sum_{\pi \in \c{S}_n^{n-1}} \IPF_{n-1}(n,\pi)=1+\sum_{j=1}^{n-1} 2^{n-j-1}=2^{n-1}.
	\]
\end{proof}
In principle this same technique can be used to compute $\IPF_{n-k}(n,\pi)$ and $\IPF_{n-k}(n)$ for any fixed $k$, though the case analysis and computations become rather complicated.  We note that one can prove $\IPF_{n-1}(n)=2^{n-1}$ more directly by observing that $(C_1,\ldots,C_n)$ will be an $(n-1)$-interval parking function if and only if $C_n=[2,n]$ and $C_i$ is $[1,n-1]$ or $[2,n]$ for all other $i$.

Before proving Corollary~\ref{C-2}, we give an enumeration result for permutations in $\c{S}_n^2$ with a given number of ascents.  We adopt the convention $\Eul{0}{k}=0$ for $k>0$, $\Eul{0}{0}=1$, and $\Eul{n}{-1}=0$.

\begin{lem}\label{L-Count}
	For all $n$ and $k$ with $n\ge 1$ and $0\le k\le n-1$, let $\c{S}_{n,k}^+$ be the set of permutations of size $n$ which have $\pi_{n-1}<\pi_n$ and which have exactly $k$ ascents.  If $P(n,k):=|\c{S}_{n,k}^+|$, then  \[P(n,k)=(n-k)\Eul{n-1}{k-1}.\]
\end{lem}
We note that this result is implicitly proven in \cite{S}, but for completeness we include the full proof here. To prove this, we recall the following recurrence for Eulerian numbers, which is valid for all $n,k\ge 1$  \cite{Wolf}.

\begin{align}\label{E-Eulerian}
\Eul{n}{k}=(k+1)\Eul{n-1}{k}+(n-k)\Eul{n-1}{k-1}.
\end{align}

\begin{proof}
	The result is true for $k=0$, so assume that we have proven the result up to $k\ge 1$.  For any fixed $k$ the result is true for $n=1$, so assume the result has been proven up to $n\ge 2$. 
	To help us prove the result, we define $\c{S}_{n,k}^-$ to be the set of permutations which have $\pi_{n-1}>\pi_n$ and which have exactly $k$ ascents.  Define $M(n,k):=|\c{S}_{n,k}^-|$.  By construction we have
	\begin{equation}\label{E-MP}
	P(n,k)+M(n,k)=\Eul{n}{k}.
	\end{equation}
	
	Define the map $\phi:\c{S}_{n,k}^+\to \c{S}_{n-1}$ by sending $\pi \in \c{S}_{n,k}^+$ to the word obtained by removing the letter $1$ from $\pi$ and then decreasing the value of each letter by 1.  For example, $\phi(32514)=2143$.  We wish to determine the image of $\phi$.  Let $\pi$ be a permutation in $\c{S}_{n,k}^+$, and let $i$ denote the position of $1$ in $\pi$.  Note that $i\ne n$ since $\pi$ ends with an ascent.  If $\pi_{i-1}<\pi_{i+1}$ with $1<i<n$, then $\phi(\pi)$ will continue to have $k$ ascents and end with an ascent, so $\phi(\pi)\in \c{S}_{n-1,k}^+$.  If $i=1$ or $\pi_{i-1}>\pi_{i+1}$ with $1<i<n-1$, then $\phi(\pi)\in \c{S}_{n-1,k-1}^+$.  If $i=n-1$ and $\pi_{n-2}>\pi_n$, then $\phi(\pi)\in \c{S}_{n-1,k-1}^-$.
	
	It remains to show how many times each element of the image is mapped to by $\phi$.  If $\pi\in \c{S}_{n-1,k}^+$, then $1$ can be inserted into $\pi$ in $k$ ways to obtain an element of $\c{S}_{n,k}^+$ (it can be placed between any of the $k$ ascents $\pi_i<\pi_{i+1}$).  If $\pi\in \c{S}_{n-1,k-1}^+$, then 1 can be inserted into $\pi$ in $n-k$ ways (it can be placed at the beginning of $\pi$ or between any of the $n-1-k$ descents $\pi_i>\pi_{i+1}$).  If $\pi\in \c{S}_{n-1,k}^-$, then $1$ must be inserted in between $\pi_{n-1}>\pi_n$ in order to have the word end with an ascent.  With this and the inductive hypothesis, we conclude that
	\begin{align}
	P(n,k)&=kP(n-1,k)+(n-k)P(n-1,k-1)+M(n-1,k-1)\nonumber\\ 
	&=k(n-k-1)\Eul{n-2}{k-1}+(n-k)^2\Eul{n-2}{k-2}+M(n-1,k-1).\label{E-AEul}
	\end{align}
	
	By using \eqref{E-MP}, the inductive hypothesis, and \eqref{E-Eulerian}; we find \begin{align*}M(n-1,k-1)&=\Eul{n-1}{k-1}-P(n-1,k-1)\\ &=\Eul{n-1}{k-1}-(n-k)\Eul{n-2}{k-2}\ \\ &=k\Eul{n-2}{k-1}.\end{align*}
	
	Plugging this into \eqref{E-AEul} and using \eqref{E-Eulerian} gives \[P(n,k)=(n-k)\l(k\Eul{n-2}{k-1}+(n-k)\Eul{n-2}{k-2}\r)=(n-k)\Eul{n-1}{k-1},\]
	as desired.
\end{proof}

\begin{proof}[Proof of Corollary~\ref{C-2}]
	By Theorem~\ref{T-k} we have, after evaluating terms which are automatically 1, 
	\[\IPF_2(n,\pi)=\prod_{i=2}^{n-1} \min\{a_i(\pi),2\}.\]
	Note that $a_i(\pi)\ge 2$ if and only if $\pi_{i-1}<\pi_i$.  There are exactly $\asc(\pi)-1$ different $i$ with $2\le i<n$ satisfying this, where we subtract 1 since $\pi\in \c{S}_n^2$ implies that there is always an ascent at position $n-1$.  We conclude the first result.
	
	For the second result, we sum $\IPF_2(n,\pi)$ over all $\pi\in \c{S}_n^2$.  Each term contributes $2^{\asc(\pi)-1}$, so we conclude the result by Lemma~\ref{L-Count} after noting that $\Eul{n-1}{-1}=0$.
\end{proof}

Corollary~\ref{C-2} shows that, for $n\ge 2$, $\IPF_2(n)$ is equal to the number of connected threshold graphs on $n$ vertices \cite{S}.  This can be proven bijectively from essentially the same proof as in \cite{S}, but for brevity we omit the details.  The formulas for $k=3$ and $k=n-2$ seem complicated (though the latter can be put into a closed form), and as of this writing neither sequence appears in the OEIS.

Before proving our enumeration results for interval parking functions, we first directly enumerate the number of parking functions with a given outcome. In what follows we treat parking functions as subset parking functions $(C_1,\ldots,C_n)$ where $C_i=[c_i,n]$ for some $1\le c_i\le n$.   Define a partial parking function $(C_1,\ldots,C_m)$ analogous to how we defined $\c{L}$-partial parking functions.  We immediately have the following.

\begin{lem}\label{L-P1}
	Let $1\le m\le n$ and $\pi \in \c{S}_n$.  $\c{C}$ is a parking function with outcome $\pi$ if and only if $(C_1,\ldots,C_m)$ is a partial interval parking function with outcome $\pi^{(m)}$ for all $m$.
\end{lem}

We also have an analog of Lemma~\ref{L-k2}.
\begin{lem}\label{L-P2}
	Let $1\le m\le n$ and $\pi \in \c{S}_n$. Let $(C_1,\ldots,C_{m-1})$ be a partial parking function with outcome $\pi^{(m-1)}$ and let $p=\pi_m^{-1}$.  Then $(C_1,\ldots,C_{m})$ is a partial parking function with outcome $\pi^{(m)}$ if and only if $C_m=[r,n]$ with $p-a_p(\pi)+1\le r\le p$.
\end{lem}
\begin{proof}
Assume $(C_1,\ldots,C_{m})$ is such a partial parking function with $C_m=[r,n]$ for some $r$.  We need $r\le p$ so that this set contains $j$.  Further, we require every $x\in [r,p]$ to satisfy $\pi_x<m$, otherwise $p$ will not be the smallest element of $C_m\sm \{\pi_i^{-1}:i<m\}$, which would contradict $(C_1,\ldots,C_m)$ having outcome $\pi^{(m)}$.  By definition this will not be the case if $r<p-a_p(\pi)+1$, so $r\ge p-a_p(\pi)+1$.  We conclude that $r$ satisfies the desired inequalities.

Conversely, assume $C_m=[r,n]$ has $r$ satisfying these inequalities.  Because $p-a_p(\pi)+1\ge 1$ we have $C_m\sub [n]$.  Again by definition of $a_p(\pi)$ these inequalities imply that $p$ is the smallest element of $C_m\sm \{\pi_i^{-1}:i<m\}$, so this gives the desired partial parking function.
\end{proof}

\begin{prop}\label{P-Park}
	\[\PF(n,\pi)=\prod_{i=1}^na_i(\pi).\]
\end{prop}
\begin{proof}
	Assume one has chosen $C_1,\ldots,C_{i-1}$ so that $(C_1,\ldots,C_{i-1})$ is a partial parking function with outcome $\pi^{(i-1)}$.  There are $a_{\pi^{-1}_i}(\pi)$ choices for $C_i$ to make  $(C_1,\ldots,C_{i})$ a partial parking function with outcome $\pi^{(i)}$ by Lemma~\ref{L-P2}.  Every parking is obtained this way by Lemma~\ref{L-P1}, so  taking the product over all these values and reindexing gives the desired result.
\end{proof}
We use Theorem~\ref{T-TechI} to prove Theorem~\ref{T-I}, and to do so we require the following lemma.

\begin{lem}\label{L-Sum}
	For any $n\ge 1$ and $\pi\in \c{S}_n,$ \[\sum_{k=1}^n b_i(\pi,k)=a_i(\pi)(n-i+1).\]
\end{lem}
\begin{proof}
	Throughout this proof we use that the ``otherwise'' case in the definition of $b_i(\pi,k)$ can be written as $n-i+1+a_i(\pi)-\max\{a_i(\pi),k\}$.
	
	We first consider the case $a_i(\pi)\le n-i$ and split the sum into two parts.  For $a_i(\pi)<n-i$ we have \begin{align}\sum_{k=1}^{n-i} b_i(\pi,k)&=\sum_{k=1}^{n-i} \min(a_i(\pi),k)=\sum_{k=1}^{a_i(\pi)}k+\sum_{k=a_i(\pi)+1}^{n-i} a_i(\pi)\nonumber\\ &={a_i(\pi)+1\choose 2}+(n-i-a_i(\pi))a_i(\pi)\label{E-B11}.\end{align}
	If $a_i(\pi)=n-i$ the same formula holds by essentially the same reasoning.
	
	If $k>n-i+a_i(\pi)$ we have $b_i(\pi,k)=0$, so the rest of the sum is \begin{align}\sum_{k=n-i+1}^{n-i+a_i(\pi)} b_i(\pi,k)&=(n-i+a_i(\pi)+1)a_i(\pi)-\sum_{k=n-i+1}^{n-i+a_i(\pi)} \max\{a_i(\pi),k\}\nonumber\\ &=(n-i+a_i(\pi)+1)a_i(\pi)-\sum_{k=n-i+1}^{n-i+a_i(\pi)} k,\label{E-B21}\end{align}
	where we used our assumption $a_i(\pi)\le n-i<k$ in this last equality.  Using the identity $\sum_{k=x+1}^{x+y} k=xy+{y+1\choose 2}$, we conclude that \eqref{E-B21} equals
	\[
		(a_i(\pi)+1)a_i(\pi)-{a_i(\pi)+1\choose 2}.
	\]
	Adding this to \eqref{E-B11} gives the desired result.
	
	Now assume $a_i(\pi)\ge n-i+1$.  In this case we have
	\begin{align}\sum_{k=1}^{n-i} b_i(\pi,k)&=\sum_{k=1}^{n-i} \min(a_i(\pi),k)=\sum_{k=1}^{n-i}k\nonumber\\ &={n-i+1\choose 2}\label{E-B12}.\end{align}
	
	The rest of the sum is 
	\begin{align}\sum_{k=n-i+1}^{n-i+a_i(\pi)} b_i(\pi,k)&=(n-i+a_i(\pi)+1)a_i(\pi)-\sum_{k=n-i+1}^{n-i+a_i(\pi)} \max\{a_i(\pi),k\}\nonumber\\ &=(n-i+a_i(\pi)+1)a_i(\pi)-a_i(\pi)(a_i(\pi)-n+i)-\sum_{k=1+a_i(\pi)}^{n-i+a_i(\pi)} k\nonumber\\ &=(2n-2i+1)a_i(\pi)-{n-i+1\choose 2}-a_i(\pi)(n-i).\label{E-B22}\end{align}
	Adding \eqref{E-B12} and \eqref{E-B22} gives the desired result.
\end{proof}

\begin{proof}[Proof of Theorem~\ref{T-I}]
	Observe that interval parking functions are exactly $\c{K}$-interval parking functions with $K_i=[n]$ for all $i$, so by Theorem~\ref{T-TechI}, Lemma~\ref{L-Sum}, and Proposition~\ref{P-Park}; we have
	\[\IPF(n,\pi)=\prod_{i=1}^n a_i(\pi)(n-i+1)=n!\prod_{i=1}^n a_i(\pi)=n!\cdot \PF(n,\pi).\]
	
	Using \eqref{E-Park}, we find 
	\[\IPF(n)=\sum_{\pi\in \c{S}_n} \IPF(n,\pi)=n!\sum_{\pi \in \c{S}_n}\PF(n,\pi)=n!\cdot \PF(n)=n!\cdot (n+1)^{n-1}.\]
\end{proof}

\section{Acknowledgments}
The author was fortunate to have many fruitful discussions about this topic at the Graduate Research Workshop in Combinatorics.  In particular we would like to thank Emma Christensen, Ryan DeMuse, Sean English, Jeremy Martin, Puck Rombach, Mike Ross, and Mei Yin.

This material is based upon work supported
by the National Science Foundation Graduate Research Fellowship under Grant No. DGE-1650112.  This work was completed in part at the 2019 Graduate Research Workshop in Combinatorics, which was supported in part by NSF grant \#1923238, NSA grant \#H98230-18-1-0017,  a generous award from the Combinatorics Foundation, and Simons Foundation Collaboration Grants \#426971 (to M. Ferrara), \#316262 (to S. Hartke) and \#315347 (to J. Martin).

\end{document}